\documentclass[9pt,reqno]{amsproc}

\usepackage{amsmath}
\usepackage{amsthm}
\usepackage{amssymb}

\usepackage{enumerate}

\usepackage{mathpazo} %
\usepackage[scaled]{helvet} %
\usepackage{courier} %
\normalfont
\usepackage[T1]{fontenc}

\usepackage{url}
\usepackage{wasysym}

\usepackage{tikz-cd}

\usepackage{graphicx}

\usepackage{mathtools}

\usepackage{wrapfig}

\newcommand{\blackboardbold}[1]{\ensuremath{\mathbf{#1}}}

\def\ZZ{\blackboardbold{Z}}
\def\QQ{\blackboardbold{Q}}

\def\FF{\blackboardbold{F}}
\def\GG{\blackboardbold{G}}

\newcommand{\mathtext}[1]{\ensuremath{\mathrm{#1}}}

\newcommand{\HP}{\ensuremath{\blackboardbold{HP}}}
\newcommand{\OP}{\ensuremath{\blackboardbold{CaP}}}
\newcommand{\MO}{\ensuremath{\mathtext{MO}}}
\newcommand{\MU}{\ensuremath{\mathtext{MU}}}

\newcommand{\BP}{\ensuremath{\mrm{BP}}}

\newcommand{\bb}[1]{\ensuremath{\langle #1 \rangle}}

\newcommand{\MString}{\MO\bb8}

\newcommand{\MSpin}{\mathtext{MSpin}}

\newcommand{\tmf}{\mathtext{tmf}}

\newcommand{\mrm}[1]{\mathrm{#1}}

\newcommand{\inv}[2]{\big[\tfrac#1#2\big]\!}

\renewcommand{\H}{\mathtext{H}}
\newcommand{\K}{\mathtext{K}}
\newcommand{\Tot}{\mathtext{Tot}}

\newcommand{\iso}{\ensuremath{\cong}}

\newcommand{\tensor}{\otimes}

\newcommand{\Hom}{\mathtext{Hom}}

\DeclareMathOperator{\cofiber}{cofiber}
\DeclareMathOperator{\cone}{cone}
\DeclareMathOperator{\colim}{colim}
\DeclareMathOperator{\hocolim}{hocolim}

\theoremstyle{plain}
\newtheorem*{theorem*}{Theorem}
\newtheorem{theorem}{Theorem}
\newtheorem*{proposition*}{Proposition}
\newtheorem{proposition}[theorem]{Proposition}
\newtheorem{corollary}[theorem]{Corollary}
\newtheorem*{corollary*}{Corollary}

\newtheorem*{lemma*}{Lemma}

\newtheorem*{exercise*}{Exercise}

\newtheorem*{conjecture*}{Conjecture}

\newtheorem*{question*}{Question}

\theoremstyle{definition}

\newtheorem*{definition*}{Definition}

\newtheorem*{example*}{Example}
\newtheorem*{examples*}{Examples}
\newtheorem*{claim*}{Claim}

\newtheorem*{remark*}{Remark}
\theoremstyle{plain}

{

\newcommand{\xto}{\xrightarrow}

\allowdisplaybreaks[1]

\frenchspacing

\title{tmf Is Not a Ring Spectrum Quotient of String Bordism}

\author{Carl McTague}
\email{mctague@math.jhu.edu}
\urladdr{http://www.mctague.org/carl}
\address{Mathematics Department, University of Rochester, Rochester, NY 14627, USA}
\address{Mathematics Department, Johns Hopkins University, Baltimore, MD 21218, USA}

\begin{document}

\begin{abstract}
This paper shows that $\mathrm{tmf}{\big[\tfrac16\big]\!}$ is not a ring spectrum quotient of $\mathrm{MO}\langle8\rangle{\big[\tfrac16\big]\!}$. In~fact, for any prime $p>3$ and any sequence $X$ of homogeneous elements of $\pi_*\mathrm{MO}\langle8\rangle$, the $\pi_*\mathrm{MO}\langle8\rangle$-module $$\pi_*\big(\mathrm{MO}\langle8\rangle_{(p)}/X\big)$$ is not (even abstractly) isomorphic to $\pi_*\mathrm{tmf}_{(p)}$.

It does so by showing that, for any commutative ring spectrum $R$ and any sequence $X$ of homogeneous elements of $\pi_*(R)$, there is an isomorphism of graded $\mathbf{Q}$-vector spaces $$\pi_*(R/X)\otimes\mathbf{Q} \cong \mathrm{H}_*(\mathrm{Tot}(\mathrm{K}(X)))\otimes\mathbf{Q},$$ where the right-hand side is the rational homology of the (total) Koszul complex of $X$, which is strictly bigger than $\pi_*(R)/(X)\otimes\mathbf{Q}$ unless $X$ is a $\pi_*(R)\otimes\mathbf{Q}$-quasi-regular sequence. The result then follows from the fact that the kernel of the $p$-local Witten genus cannot be generated by a $\pi_*\mathrm{MO}\langle8\rangle\otimes\mathbf{Q}$-quasi-regular sequence.
\end{abstract}

\maketitle

\section{Introduction}

Given a commutative ring spectrum $R$, an $R$-module $M$, and a \emph{sequence} $X=x_1,x_2,\dots$ of homogeneous elements of $\pi_*(R)$, one can construct the quotient $R$-module $M/XM$, by iteratively forming cofibers
\begin{align}
M\big/(x_1,\dots,x_{n+1})M = \cofiber \left[ \Sigma^{\deg x_{n+1}} M\big/(x_1,\dots,x_n)M \xto{x_{n+1}} M\big/(x_1,\dots,x_n)M \right] \label{eqn:cofiber}
\end{align}
and, if the sequence is infinite, passing to their telescope
$$
M\big/(x_1,x_2,x_3,\dots)M = \hocolim_n \left[ M\big/(x_1,\dots,x_n)M \right].
$$

When $M=R$, the resulting $R$-module, denoted $R/X$, often comes equipped with a homotopy commutative $R$-algebra structure. For example, if $R$ is a commutative $S$-algebra with $\pi_i(R)=0$ for $i$ odd, and if $X$ is a sequence of non zerodivisors such that $\pi_*(R/X)$ is concentrated in degrees congruent to 0 modulo 4, then $R/X$, regarded as an object in the \emph{homotopy category} of $R$-modules, admits a unique (up to canonical isomorphism) ring object structure, and it is commutative and associative \cite[Thm.~V.3.2]{elmendorf97}. To emphasize: although $R$ is required to have a commutative ring spectrum structure, the quotient $R/X$ is only guaranteed to admit a homotopy commutative one.

\emph{If $X$ is regular} in the sense that each element $x_{n+1}$ of $X$ is not a zerodivisor for $\pi_*(R)\big/(x_1,\dots,x_n)$, then the long exact homotopy sequence for each cofiber breaks into short exact sequences, and  $$\pi_*(R/X)\iso\pi_*(R)\big/(X),$$ where $(X)$ is the \emph{ideal} generated by the elements of $X$.%

\smallskip

This, together with a theory of localization, makes possible a one-page construction (from $\MU$ as a commutative ring spectrum, and for $p>2$)\footnote{When $p=2$, $\pi_*(\mathrm{BP})$ is not concentrated in degrees congruent to 0 modulo 4. Nevertheless, $\BP$ satisfies weaker conditions \cite[Thm.~2.7]{strickland-1999} ensuring that it admits a homotopy commutative $\MU_{(2)}$-algebra structure \cite[Prop.~2.9]{strickland-1999}; the same is true of $\BP\bb{n}$ and $\mathrm{E}(n)$, provided the generators $v_i$ are chosen carefully \cite[Prop.~2.10]{strickland-1999}.} of a rogue's gallery of spectra (listed with coefficient ring),
\begin{align*}
  \BP &\quad \ZZ_{(p)}[v_1,v_2,\dots] &
  \BP\bb{n} &\quad \ZZ_{(p)}[v_1,\dots,v_n] \\
  \mathrm{E}(n) &\quad \ZZ_{(p)}[v_1,\dots,v_n,v_n^{-1}] &
  \mathrm{P}(n) &\quad \FF_p[v_n,v_{n+1},\dots] \\
  \mathrm{B}(n) &\quad \FF_p[v_n^{-1},v_n,v_{n+1},\dots] &
  \mathrm{k}(n) &\quad \FF_p[v_n] \\
  \mathrm{K}(n) &\quad \FF_p[v_n,v_n^{-1}],
\end{align*}
all equipped with unique (up to canonical isomorphism) homotopy commutative $\MU$-algebra structures \cite[\S V.4]{elmendorf97}.  (Previously these spectra had been constructed using Baas-Sullivan theory or the Landweber exact functor theorem, and laboriously shown to admit ring spectrum structures \cite{sullivan69,baas-1973,mironov-1975,shimada-yagita-1976,morava-1979}.)

It also makes possible a one-line construction of Landweber-Ravenel-Stong elliptic cohomology \cite{lrs-1993} (from $\MSpin$ as a commutative ring spectrum),
$$
  \mathrm{ell}\inv12 = \MSpin\inv12 / \big(\text{kernel of the Ochanine elliptic genus}\big),
$$
similarly equipped with a unique (up to canonical isomorphism) homotopy commutative $\MSpin$-algebra structure.
(In fact, since the kernel of the Ochanine elliptic genus is the ideal consisting of $\HP^2$-bundles, one may write
$$
  \mathrm{ell}\inv12 = \MSpin\inv12 / \big(\text{$\HP^2$-bundles}\big)
$$
\cite{kreck-stolz-93}; cf.\ also the Conner-Floyd isomorphism for $\mathrm{ell}$ \cite{hovey-1995,hopkins-hovey-1992}.)

\smallskip

By contrast, this paper shows that it is impossible to construct the spectrum $\tmf$ of topological modular forms \cite{hopkins-1995,douglas-francis-henriques-hill-2014} in an analogous way from the string bordism ring spectrum $\MString$. (Even though
\begin{align*}
\pi_*\tmf\inv16
&\iso \pi_*(\MString)\inv16/\big(\text{kernel of the Witten genus}\big) \\
&= \pi_*(\MString)\inv16/\big(\text{$\OP^2$-bundles}\big)
\end{align*}
\cite{mctague-2014}.)

\begin{theorem}
  \label{thm:noX} For any prime $p>3$ and any sequence $X$ of homogeneous elements of $\pi_*\MString$, the $\pi_*\MString$-module
  $$
    \pi_*\big(\MString_{(p)}/X\big)
  $$
  is not (even abstractly) isomorphic to $\pi_*\tmf_{(p)}$.
\end{theorem}

As we shall see, the reason is twofold:
\begin{enumerate}[(I)]
\item For any (homotopy) commutative ring spectrum $R$ and any sequence $X$ of homogeneous elements of $\pi_*(R)$, there is an isomorphism of graded $\QQ$-vector spaces $$\pi_*(R/X)\tensor\QQ \cong \H_*(\Tot(\K(X)))\tensor\QQ,$$ where the right-hand side is the rational homology of the (total) Koszul complex of $X$ (Theorem~\ref{thm:koszul}), which is strictly bigger than $\pi_*(R)/(X)\tensor\QQ$ unless $X$ is a $\pi_*(R)\tensor\QQ$-quasi-regular sequence (Theorem~\ref{thm:bourbaki}).
\item The kernel of the $p$-local Witten genus
  $$
    \pi_*\MString_{(p)}\to\pi_*\tmf_{(p)}
  $$
cannot be generated by a $\pi_*\MString\tensor\QQ$-quasi-regular sequence (Proposition~\ref{prop:witten-quasi-regular}).
\end{enumerate}

\begin{corollary}
  \label{cor:noX}
  There do not exist $\MString$-module (let alone $\MString$-algebra) isomorphisms
  \begin{align*}
    \MString_{(p)}/X &\to \tmf_{(p)} &
    \MString\inv16/X &\to \tmf\inv16.
  \end{align*}
\end{corollary}

(If such isomorphisms existed, their underlying $\pi_*\MString$-module isomorphisms would contradict Theorem~\ref{thm:noX}.)

\smallskip

This helps to better appreciate the elaborate known constructions of $\tmf$ and its orientations, cf.~\cite{ando-2001,hopkins-2002,lurie-2009,douglas-francis-henriques-hill-2014}. It also raises the question of whether there is a general modification of, or alternative to, the ring spectrum quotient construction, one better suited to non quasi-regular sequences, which produces $\tmf$ from $\MString$ in a simple way.

\section{The Koszul Complex $\boldsymbol{\K(X)}$}

Let $A$ be a graded and commutative ring, $E$ an $A$-module, and $f:E\to A$ a map of graded $A$-modules. Recall that the \emph{Koszul complex} $\K(f)$ of $f$ is the exterior algebra $\bigwedge E$ equipped with differential
$$
  d_k(e_1\wedge\cdots\wedge e_k) = \sum_{i=1}^k (-1)^{i+1}f(e_i)\;e_1\wedge\cdots\wedge\hat e_i\wedge\cdots\wedge e_k,
$$
where $\hat e_i$ means that the wedge summand $e_i$ is omitted. $\bigwedge E$ is bigraded,
$$
  \deg(e_1\wedge\cdots\wedge e_k) = \big(k,\deg(e_1)+\cdots+\deg(e_k)\big),
$$
and the differential $d_k$ decreases the first component by 1 and preserves the second, so $\K(f)$ may be thought of as a graded complex of $A$-modules
$$
  \cdots \to \wedge^2E \xto{d_2} \wedge^1E\xto{d_1}A\to0.
$$
A (possibly infinite) sequence $X=x_1,x_2,\dots$ of homogeneous elements of $A$ defines a map of graded $A$-modules
\begin{align*}
  f:&\bigoplus_i A(-\deg x_i)\to A \\
  f\big(&\textstyle \bigoplus_i a_i\big) = \sum_i x_ia_i.
\end{align*}
Write $\K(X)$ for the Koszul complex of this map.
(Here and below, $N(m)$ denotes the shift of a graded module $N$ defined by $N(m)_k=N_{k+m}$.)

\emph{$\K(f)$ is functorial:} A map $\phi:E'\to E$ of graded $A$-modules extends to a map of bigraded $A$-algebras $\wedge\phi:\wedge E'\to\wedge E$, which defines a map of graded complexes $\K(f\circ\phi)\to\K(f)$.

Also, \emph{$\K(f)$ is exponential:} The sum of graded $A$-linear maps $f:E\to A$ and $f':E'\to A$ is a graded $A$-linear map $f\oplus f':E\oplus E'\to A$, and the isomorphism of graded $A$-modules $\bigwedge (E\oplus E') \iso \bigwedge E\tensor \bigwedge E'$ extends to an isomorphism of graded complexes
$$
  \K(f\oplus f')\iso\K(f)\tensor\K(f').
$$
In particular, if $E'=A(m)$, then the isomorphism of graded $A$-modules
$$
  \textstyle \bigwedge^k\big(E\oplus A(m)\big)\iso\bigwedge^kE\oplus\bigwedge^{k-1}E(m),
$$
induces, for any finite sequence $x_1,\dots,x_{n+1}$ of homogeneous elements of $A$, a short exact sequence of graded complexes
$$
  0 \to \K(x_1,\dots,x_n) \to\K(x_1,\dots,x_{n+1}) \to \K(x_1,\dots,x_n)(-\deg x_{n+1})[-1] \to0
$$
whose connecting homomorphism is induced by multiplication by $x_{n+1}$. (Here and below, $C[m]$ denotes the shift of a complex $(C,d_C)$ defined by $C[m]_i=C_{i+m}$ and $d_{C[1]}=(-1)^md_C$.)

Thus $\K(x_1,\dots,x_{n+1})$ is isomorphic to the mapping cone %
\begin{align}
  \label{eqn:cone}
  \K(x_1,\dots,x_{n+1}) \iso \cone\left[ \K(x_1,\dots,x_n)(-\deg x_{n+1}) \xto{x_{n+1}} \K(x_1,\dots,x_n) \right]
\end{align}
(compare \cite[Tag 0628]{stacks-project} or \cite[Cor.~1.6.13]{bruns-herzog-1993}).
Observe how similar (\ref{eqn:cone}) is to (\ref{eqn:cofiber}).

\section{$\boldsymbol{\pi_*(R\big/X)}$ is Rationally Isomorphic to $\boldsymbol{\H_*(\Tot(\K(X)))}$}

\begin{theorem}
  \label{thm:koszul}
  For any (homotopy) commutative ring spectrum $R$ and any sequence $X$ of homogeneous elements of $\pi_*(R)$, there is an isomorphism of graded $\QQ$-vector spaces
  $$
    \pi_*(R/X) \tensor\QQ \iso \H_*(\Tot(\K(X))) \tensor\QQ
  $$
  making the diagram
$$\begin{tikzcd}
  \pi_*(R) \tensor\QQ \arrow[r] \arrow[d] & \H_0(\K(X)) \tensor\QQ \arrow[d,hook] \\
  \pi_*(R/X) \tensor\QQ \arrow[r,"\iso",leftarrow] & \H_*(\Tot(\K(X))) \tensor\QQ
\end{tikzcd}$$
commute.

\begin{remark*}
  The right-hand side is the rational homology of the \emph{total} Koszul complex, so the isomorphism determines, for each integer $n$, an isomorphism of $\QQ$-vector spaces
  $$
    \pi_n(R/X) \tensor \QQ \iso \bigoplus_{i+j=n} \H_i(\K(X))_j \tensor \QQ,
  $$
  where $(-)_j$ denotes the degree-$j$ component of the graded module $(-)$.
\end{remark*}
\end{theorem}

\begin{remark*}
  Theorem~\ref{thm:koszul} does not hold integrally. In general, there is a spectral sequence $$\H_*(\K(X))\implies\pi_*(R/X)$$
 \cite[p.~185]{morava-1979}. (According to Theorem~\ref{thm:koszul}, this spectral sequence collapses tensor $\QQ$.)
\end{remark*}

\begin{proof}[Proof of Theorem~\ref{thm:koszul}]
  Let $X=x_1,x_2,\dots$ be an infinite sequence. First we use induction to argue that the result holds for each finite subsequence
$$X_n=x_1,\dots,x_n \quad (n\ge0).$$

\smallskip

  (Base case) The result holds integrally for the empty sequence $X_0$. Indeed, by definition $R/X_0=R$ and $\K(X_0)=(\bigwedge \pi_*(R),d)=\pi_*(R)[0]$, so for any integer $i$,
  $$
    \pi_i(R/X_0)=\pi_i(R) = \H_0(\K(X_0))_i = \H_0(\pi_*(R)[0])_i = \H_i(\Tot(\pi_*(R)[0])) = \H_i(\Tot(\K(X_0))).
  $$

  \medskip
  
  (Inductive step) Suppose that the conclusion holds for some subsequence $X_n$ ($n\ge0$). We argue that it then also holds for the subsequence $X_{n+1}$. Let $m=\deg(x_{n+1})$.

  According to (\ref{eqn:cone}), there is a distinguished triangle of graded complexes
  $$\begin{tikzcd}
  \K(X_n)(-m) \ar[r,"x_{n+1}"] & \K(X_n) \ar[r] & \K(X_{n+1}) \ar[r] & \K(X_n)(-m)[-1]
  \end{tikzcd}$$
  and hence a distinguished triangle of total complexes
  $$\begin{tikzcd}
    \Tot(\K(X_n))[-m] \ar[r,"x_{n+1}"] & \Tot(\K(X_n)) \ar[r] & \Tot(\K(X_{n+1})) \ar[r] & \Tot(\K(X_n))[-1-m].
  \end{tikzcd}$$
  Its long exact homology sequence looks like
  $$\hspace{-18pt}\begin{tikzcd}[column sep=tiny]
    \cdots\ar[r]& \H_{i-m}(\Tot(\K(X_n))) \ar[r]& \H_i(\Tot(\K(X_n))) \ar[r]& \H_i(\Tot(\K(X_{n+1}))) \ar[r]& \H_{i-1-m}(\Tot(\K(X_n))) \ar[r]& \cdots,
  \end{tikzcd}$$
  which looks just like the long exact homotopy sequence
$$\hspace{-17pt}\begin{tikzcd}[column sep=5.7ex]
  \cdots\ar[r]& \pi_{i-m} (R/X_n) \ar[r]& \pi_i (R/X_n) \ar[r]& \pi_i (R/X_{n+1}) \ar[r]& \pi_{i-1-m} (R/X_n) \ar[r]&\cdots
\end{tikzcd}$$
  for the cofiber sequence (\ref{eqn:cofiber}) above,
  $$\begin{tikzcd}
    \Sigma^m R/X_n \ar[r,"x_{n+1}"] & R/X_n \ar[r] & R/X_{n+1} \ar[r] & \Sigma^{1+m} R/X_n.
  \end{tikzcd}$$

  Specifically, by inductive hypothesis there are isomorphisms (solid arrows) fitting into the commutative diagram
  $$\hspace{-50pt}
    \begin{tikzcd}[column sep=tiny]
    \cdots\ar[r] & \H_{i-m}(\Tot(\K(X_n))) \tensor\QQ \ar[r] \ar[d,"\iso"] & \H_i(\Tot(\K(X_n)) \tensor\QQ \ar[r] \ar[d,"\iso"] & \H_i(\Tot(\K(X_{n+1}))) \tensor\QQ \ar[r] \ar[d,"\iso",dashed] & \H_{i-1-m}(\Tot(\K(X_n))) \tensor\QQ \ar[r] \ar[d,"\iso"] &\cdots \\
    \cdots\ar[r] & \pi_{i-m} (R/X_n) \tensor\QQ \ar[r] & \pi_i (R/X_n) \tensor\QQ \ar[r] & \pi_i (R/X_{n+1}) \tensor\QQ \ar[r] & \pi_{i-1-m} (R/X_n) \tensor\QQ \ar[r] &\cdots.
    \end{tikzcd}
  $$
  Since exact sequences of vector spaces split, there are isomorphisms (dashed arrows) making the diagram commute.

To see that these dashed isomorphisms fit together to make the diagram in the statement of the theorem commute, sum the middle squares over $i$, and use the inductive hypothesis and the fact that $\K(X_n)\to\K(X_{n+1})$ is a map of graded complexes to form the commutative diagram
  $$\begin{tikzcd}
    & \H_0(\K(X_n)) \tensor\QQ \ar[d,hook] \ar[r] & \H_0(\K(X_{n+1})) \tensor\QQ \ar[d,hook] \\
    \pi_*(R) \tensor\QQ \ar[ur, in=180, out=90] \ar[dr, in=180, out=270] & \H_*(\Tot(\K(X_n))) \tensor\QQ \ar[r] \ar[d,"\iso"] & \H_*(\Tot(\K(X_{n+1}))) \tensor\QQ \ar[d,"\iso"] \\
    & \pi_*(R/X_n) \tensor\QQ \ar[r] & \pi_*(R/X_{n+1}) \tensor\QQ.
  \end{tikzcd}$$

  \smallskip

  By induction, then, the result holds for each finite subsequence $X_n (n\ge0)$.

\medskip

  (Infinite case) The infinite case follows since
  \begin{align*}
    \pi_*(R/X)\tensor\QQ
    &\stackrel{(1)}{\iso} \pi_* \big( \hocolim_n (R/X_n) \big) \tensor\QQ \\
    &\stackrel{(2)}\iso \colim_n \big( \pi_*(R/X_n) \tensor\QQ \big) \\
    &\stackrel{(3)}\iso \colim_n \big( \H_*(\Tot(\K(X_n)))\tensor\QQ \big) \\
    &\stackrel{(4)}\iso \H_*\big(\colim_n(\Tot(\K(X_n)))\tensor\QQ\big) \\
    &\stackrel{(5)}\iso \H_*\big(\Tot(\K(X))\tensor\QQ\big).
  \end{align*}
  These isomorphisms hold by
  \begin{enumerate}
  \item[(1)] definition of $R/X$,
  \item[(2)] properties of $\hocolim$ and cocontinuity of $-\tensor\QQ$ (as it is left adjoint to $\Hom(\QQ,-)$),
  \item[(3)] the finite case proved above,
  \item[(4)] exactness of filtered colimits in the category of $A$-modules, and
  \item[(5)] cocontinuity of $\bigwedge(-)$ (as it is left adjoint to the degree-1 functor $(-)_1$). \qedhere
  \end{enumerate}
\end{proof}

\section{$\boldsymbol{\pi_*(R)\to\pi_*(R\big/X)}$ is surjective only if $\boldsymbol{X}$ is a quasi-regular sequence}

Theorem~\ref{thm:noX} could be proved using regular sequences, but it is easier to prove using \emph{quasi}-regular sequences. Let us recall how regular and quasi-regular sequences are related to the homology of the Koszul complex. %

Consider a commutative (but not necessarily graded) ring $A$, an $A$-module $N$, and an ideal $\mathfrak{X}$ of $A$. The $\mathfrak{X}$-adic topology on $N$ is the topology compatible with the group structure of $N$, for which the submodules $\mathfrak{X}^rN$ ($r\ge0$) form a fundamental system of neighborhoods of zero. This topology is separated if and only if $$\bigcap_{r\ge0}\mathfrak{X}^r=0.$$

Suppose now that $\mathfrak{X}$ is generated by a (possibly infinite) sequence $X=x_1,x_2,\dots$ of elements of $A$. Consider the graded $A$-module $\bigoplus_{r\ge0} \mathfrak{X}^rN$ and the map of graded $A$-modules
$$
a_N^X : A[X_1,X_2,\dots]\tensor_A N \to \bigoplus_{r\ge0} \mathfrak{X}^r N
$$
defined by
$$
a_N^X(P\tensor n) = P(x_1,x_2,\dots) n
$$
for $n$ in $N$ and $P$ a homogeneous polynomial in indeterminants $X_1,X_2,\dots$ (not to be confused with the sequence $X$ and its elements $x_1,x_2,\dots$). Let $\mathfrak{d}$ be the ideal of $A[X_1,X_2,\dots]$ generated by the elements $(x_iX_j-x_jX_i)$ for $1\le i<j$. Then $P(x_1,x_2,\dots)=0$ if $P\in \mathfrak{d}$, so $a_N^X$ determines a map of graded $A$-modules
$$
  \alpha^X_N : (A[X_1,X_2,\dots]/\mathfrak{d}) \tensor_A N \to \bigoplus_{r\ge0} \mathfrak{X}^r N.
$$
By tensor product with $A/\mathfrak{X}$ one obtains a map of graded $A$-modules
$$
\beta^X_N : (A/\mathfrak{X})[X_1,X_2,\dots] \tensor_A N \to \bigoplus_{r\ge0} (\mathfrak{X}^r N / \mathfrak{X}^{r+1}N).
$$
The maps $a^X_N$, $\alpha^X_N$, and $\beta^X_N$ are surjective.

\begin{theorem}[{\cite[\S9, n\textsuperscript{o}~7, Thm.~1]{bourbaki-1980}}]
  \label{thm:bourbaki}
  Consider the following conditions:
  \begin{enumerate}
  \item[(i)] $X$ is $N$-regular.
  \item[(ii)] $\H_i(\K(X)\tensor N)=0$ for all $i>0$.
  \item[(iii)] $\H_1(\K(X)\tensor N)=0$.
  \item[(iv)] $\alpha^X_N$ is bijective.
  \item[(v)] $\beta^X_N$ is bijective.
  \end{enumerate}
  We then have implications (i) $\implies$ (ii) $\implies$ (iii) $\iff$ (iv) $\implies$ (v).

  If, for $1\le i\le n$, the $A$-module $N/(x_1N+\cdots+x_nN)$ is separated in the $\mathfrak{X}$-adic topology, then conditions (i) and (v) are equivalent.
\end{theorem}

(The proof of Theorem~\ref{thm:bourbaki} given in \cite{bourbaki-1980} is for a finite sequence $X$ but applies, \emph{mutatis mutandis}, to an infinite sequence.)

\begin{definition*}
  The sequence $X$ is called \emph{$N$-quasi-regular} if it satisfies condition~(v) of Theorem~\ref{thm:bourbaki}. %
An $A$-quasi-regular sequence is simply called  \emph{quasi-regular}.
\end{definition*}

In short, every $N$-regular sequence $X$ satisfies $\H_1(\K(X)\tensor N)=0$, and every sequence $X$ satisfying $\H_1(\K(X)\tensor N)=0$ is $N$-quasi-regular.

\bigskip

Theorems~\ref{thm:koszul} and \ref{thm:bourbaki} together yield:

\begin{corollary}
  For any (homotopy) commutative ring spectrum $R$ and any sequence $X$ of homogeneous elements of $\pi_*(R)$, the homomorphism
  $$\pi_*(R)\to\pi_*(R/X)$$
  is surjective only if $X$ is $\pi_*(R)\tensor\QQ$-quasi-regular.
 \end{corollary}

\begin{proof}
  If $\pi_*(R)\to\pi_*(R/X)$ is surjective, then $\pi_*(R)\tensor\QQ\to\pi_*(R/X)\tensor\QQ$ is too. By Theorem~\ref{thm:koszul}, then, $\H_1(\K(X))\tensor\QQ=0$. And by Theorem~\ref{thm:bourbaki}, $X$ must be $\pi_*(R)\tensor\QQ$-quasi-regular.
 \end{proof}

\smallskip

Note that, since $\MString$ is connective, Theorem~\ref{thm:noX} could be proved using regular (rather than quasi-regular) sequences, by the following result.

\begin{corollary}[{of Theorem~\ref{thm:bourbaki}, \cite[\S9, n\textsuperscript{o}~7, Cor.~2]{bourbaki-1980}}]
  \label{cor:positive}
  Suppose $A$ is a positively graded ring, $N$ is a graded, bounded-below $A$-module, and $X$ is a sequence of homogeneous elements of $A$ of degree $>0$. Then conditions (i) and (v) of Theorem~\ref{thm:bourbaki} are equivalent.
\end{corollary}

\noindent (The hypotheses ensure that the $\mathfrak X$-adic topology on each of the $A$-modules $M/(x_1M+\cdots+x_nM)$ is separated.)

\section{The kernel of the $\boldsymbol{p}$-local Witten genus cannot be generated by a quasi-regular sequence}

\begin{proposition}
  \label{prop:witten-quasi-regular}
  For $p>3$, the kernel of the $p$-local Witten genus
  $$
    \phi_W\tensor\ZZ_{(p)} : \pi_*\MString_{(p)}\to\pi_*\tmf_{(p)}
  $$
  cannot be generated by a $\pi_*\MString\tensor\QQ$-quasi-regular sequence. (In particular, it cannot be generated by a quasi-regular sequence.)
\end{proposition}

\begin{proof}
  Denote the kernel $\mathfrak K$, and  suppose it was generated by a $\pi_*\MString\tensor\QQ$-quasi-regular sequence $X=x_1,x_2,\dots$.

  Rationally, the Witten genus is a surjective homomorphism of polynomial rings whose generators may be chosen so that
  \begin{align*}
    \pi_*(\MString)\tensor\QQ &\iso \QQ[y_2,y_3,y_4,\dots] \\
    \pi_*(\tmf)\tensor\QQ &\iso \QQ[\GG_2,\GG_3] \\
    \ker \phi_W \tensor \QQ &= (y_4,y_5,y_6,\dots)
  \end{align*}
  where $\deg(y_i)=\deg(\GG_i)=4i$ (see \cite{hovey-2008}, \cite[Prop.~4.4]{bauer-2008}, and \cite[Thm.~6.25]{hopkins-2002}).
It follows that
  $$
    (\mathfrak K^1/\mathfrak K^2)_n \tensor \QQ \iso
    \begin{cases}
      \QQ & \text{if $n=4i\ge16$} \\
      0 & \text{otherwise.}
    \end{cases}
  $$
  Since $X$ is a $\pi_*\MString\tensor\QQ$-quasi-regular sequence generating $\mathfrak K\tensor\QQ$, the map
  $$
    \underbrace{\big(\pi_*\MString/\mathfrak K\big)[X_1,X_2,\dots] \tensor \QQ}_{\QQ[\GG_2,\GG_3][X_1,X_2,\dots]} \to \bigoplus_{r\ge0} \big(\mathfrak K^r/\mathfrak K^{r+1}\big) \tensor \QQ
  $$
  induced by the substitution $X_i\mapsto\overline x_i \in \mathfrak K^1/\mathfrak K^2$ is by definition bijective. \emph{It follows that $X$ has precisely one element of degree $4i$, for each $i\ge4$, and no other elements.}

\smallskip

Let $\mathfrak A$ be the augmentation ideal of $\pi_*\MString_{(p)}$. Applying the snake lemma to the commutative diagram of $\pi_*\MString_{(p)}$-modules
$$\begin{tikzcd}
  0 \ar[r] & \mathfrak K^2 \ar[r] \ar[d] & \mathfrak A^2 \ar[r] \ar[d] & \mathfrak A^2/\mathfrak K^2 \ar[r] \ar[d] & 0 \\
  0 \ar[r] & \mathfrak K^1 \ar[r] & \mathfrak A^1 \ar[r] & \mathfrak A^1/\mathfrak K^1 \ar[r] & 0
  \end{tikzcd}$$
(in which the rows are exact) produces an exact sequence
$$\begin{tikzcd}
  \mathfrak K^1/\mathfrak K^2 \ar[r] & \mathfrak A^1/\mathfrak A^2 \ar[r] & \mathfrak{A}^1/(\mathfrak K^1+\mathfrak A^2) \ar[r] & 0.
\end{tikzcd}$$
Since
$$
  \pi_*\MString_{(p)}/\mathfrak K\iso\ZZ_{(p)}[\GG_2,\GG_3]
$$
(again by \cite[Prop.~4.4]{bauer-2008} and \cite[Thm.~6.25]{hopkins-2002}, see also \cite{mctague-2014}), it follows that
$$
  \mathfrak{A}^1/(\mathfrak K^1+\mathfrak A^2) \iso \ZZ_{(p)}\{\GG_2,\GG_3\}
$$
and therefore that $\mathfrak K^1/\mathfrak K^2\to\mathfrak A^1/\mathfrak A^2$ is surjective in degrees $>12$.

According to Hovey's calculations \cite{hovey-2008} (see also \cite[Thm.~4]{mctague-2014})
\begin{align*}
  (\mathfrak A^1/\mathfrak A^2)_n &\iso
                                \begin{cases}
                                  \ZZ_{(p)} \oplus \ZZ/p & \text{if $n=2(p^i+p^j)$ for some $0<i<j$} \\
                                  \ZZ_{(p)} & \text{if $n=4m\ge8$ but $n\ne2(p^i+p^j)$ for any $0<i<j$} \\
                                  0 & \text{otherwise.}
                                \end{cases}
\end{align*}
Since no single element can generate $\ZZ_{(p)}\oplus\ZZ/p$ as a $\ZZ_{(p)}$-module, it follows that the map
$$
    \big(\underbrace{\pi_*\MString_{(p)}/\mathfrak K}_{\ZZ_{(p)}[\GG_2,\GG_3]}\big)[X_1,X_2,\dots] \to \bigoplus_{r\ge0} \big(\mathfrak K^r/\mathfrak K^{r+1}\big)
$$
  induced by the substitution $X_i\to\overline x_i\in\mathfrak K^1/\mathfrak K^2$ cannot be surjective, contradicting the assumption that $X$ generates $\mathfrak K$.
\end{proof}

\section*{Acknowledgments}

Thanks to Jack Morava, John Greenlees, Nitu Kitchloo, Burt Totaro, Doug Ravenel, John Lind, and Doug Haessig for helpful conversations.

\end{document}